\documentclass[a4paper]{article}
\usepackage{graphicx} 
\usepackage{amsmath}
\usepackage{amsthm}
\usepackage{amsfonts,  amscd, bm, enumerate}
\usepackage{amssymb}
\usepackage{verbatim}
\usepackage{stmaryrd}

\usepackage{url}

\newtheorem{theorem}{Theorem}
\newtheorem{corollary}{Corollary}
\newtheorem{lemma}{Lemma}
\newtheorem{definition}{Definition}

\theoremstyle{remark}
\newtheorem{remark}{Remark}

\def\la{\langle}
\def\ra{\rangle}
\def\as{a.s.}
\def\F{\mathcal{F}}
\def\CB{\mathcal{B}}
\def\CF{\mathcal{F}}

\def\CE{\mathcal{E}}
\def\FE{\mathfrak{E}}
\def\CH{\mathcal{H}}
\def\CA{\mathcal{A}}

\def\CC{\mathcal{C}}

\def\bR{\mathbb{R}}

\def\bN{\mathbb{N}}

\def\bP{\mathbb{P}}

\def\q{\qquad}
\def\loc{\mathrm{loc}}
\newcommand{\uPi}{\Pi}
\def\extreal{\overline{\bR}}

\begin{document}
\title{A martingale representation theorem for a class of jump processes.}
\author{Samuel N. Cohen\thanks{Research supported by the Oxford--Man Institute for Quantitative Finance.}\\Mathematical Institute, University of Oxford\\samuel.cohen@maths.ox.ac.uk}
\date{\today}

\maketitle
\begin{abstract}
We give a bare-hands approach to the martingale representation theorem for integer valued random measures, which allows for a wide class of infinite activity jump processes, as well as all processes with well-ordered jumps. 

Keywords: Martingale representation theorem, jump processes, random measures.

MSC: 60G57, 60G44
\end{abstract}
\section{Introduction}\label{sec:MRT}
A key result in the theory of martingales is the ``Martingale Representation Theorem''. This allows us to say that all martingales, in a certain space, can be represented in terms of stochastic integrals. These results underly much work in mathematical finance (through hedging arguments, see, for example, F\"ollmer and Schied \cite[Theorem 5.38]{Follmer2002}), are the basis for the theory of BSDEs (El Karoui, Peng and Quenez \cite{El1997} gives a review, a general representation is used in \cite{Cohen2010}), and lead to Wiener-like chaos expansions, which appear in the theory of Malliavin calculus (see, for example, Di Nunno, {\O}ksendal and Proske \cite{DiNunno2009}). 

The martingale representation theorem is commonly known in a Brownian setting (where it is due to It\=o \cite{Ito1951}), and has been well studied in the setting of L\`evy processes (see, for example Nualart and Schoutens \cite{Nualart2000}). Elegant abstract results in general spaces were obtained by Kunita and Watanabe \cite{Kunita1967} and Davis and Varaiya \cite{Davis1974}. A significant general result in this area is the Jacod--Yor theorem (see \cite[Section III.4]{Jacod2003} for a presentation), which relates martingale representations with the (non-)existence of equivalent martingale measures. A review of the existing theory and approaches is given by Davis \cite{Davis2005}.

In a series of papers in the mid-1970s, beginning with the work of Chou and Meyer \cite{Chou1975}, martingale representations for jump processes were considered directly, using a `bare-hands' approach to obtaining a representation. Relevant papers include Chou and Meyer \cite{Chou1975}, Davis \cite{Davis1976} and Elliott \cite{Elliott1976, Elliott1977}.  In this approach, which reached its most general point in Elliott \cite{Elliott1976}, the jumps are assumed to be generated by a right-constant process, without any substantial assumptions in terms of memory, or whether the jump times admit a density. However, by starting with a right-constant process, one implicitly assumes that the jumps are well ordered (in the usual ordering through time). Therefore, these results cannot subsume the case of a pure-jump L\`evy process, which may have infinitely many jumps on every interval.

In this paper, we will see how this bare-hands approach can be applied to a wide class (but not all) jump processes. We will show that a martingale representation result, in terms of compensated integer valued random measures, can be obtained for those random measures we call `parochial' -- a class which includes infinite activity L\`evy processes and also all processes with well ordered jumps (see Section \ref{sec:Parochial}). We also give a counterexample to the proposition that all integer-valued random measures yield a martingale representation result.

The paper is structured in the following way. We first give a review of the relevant results and notation for random measures. We then give a version Chou and Meyer's construction of the martingale representation for a single jump process, followed by Elliott's extension to the case of well ordered jumps. We then define the class of parochial random measures, and show how Elliott's result can be extended to these cases. This approach to obtaining martingale representations for infinite activity process has not been pursued before. This is followed by a counterexample, which highlights why our assumptions are needed. The paper concludes by introducing continuous processes to our setting, and showing how a martingale representation result for jump processes can be extended to these cases.

\section{A quick review of random measures}
We here summarize some key definitions and results in the theory of random measures. These can be found in \cite{Elliott1982} or \cite{Jacod2003}, and are presented without proof. We suppose we are working on a probability space $(\Omega,\CF,\bP)$
 which has a complete, right continuous filtration $\{\CF_t\}_{t\ge
 0}.$ We also have an auxillary Blackwell space  $(\CE, \FE)$ in which we observe `jumps'; however, such generality is typically not required as many applications are when $\CE\subset \bR^n$ or $\CE\subset \bR^\bN.$

Write
$$
\tilde\Sigma_p =\Sigma_p\otimes \FE,
$$
where $\Sigma_p$ is the predictable $\sigma$-algebra, so $\tilde\Sigma_p$ is a $\sigma$-algebra on $\Omega\times[0,\infty]\times\CE$.

\begin{definition}\label{Def15.20}
A \emph{nonnegative random measure} $\mu$ is a family $\{\mu(\omega,\cdot),\, \omega  \in \Omega  \}$ of $\sigma$-finite measures on $\big([0,\infty]\times \CE,\CB([0,\infty])\otimes \FE\big).$ A function which can be written as the difference of two nonnegative random measures is called a random measure.
\end{definition}

Suppose $W:\tilde \Omega  \to\bR$ is a map such that each section $W(\omega ,\cdot):[0,\infty]\times \CE\to \bR$ is a Borel measurable function.  Then write
$$
(W*\mu)_t(\omega  ) =\int_{[0,t]\times \CE} W(\omega  ,s,x)\mu (\omega  ,ds\times dx),
$$
if this integral exists (possibly equal to $+\infty  $ or $-\infty$).

If $(W*\mu)_t$ exists for all $t\in [0,\infty[$ and is measurable with respect to $\omega$, one can talk of the {\it process} $W*\mu  .$ This process has jumps 
\[\Delta (W*\mu)_t = \int_{\CE} W(\omega, t, x)\mu(\omega, \{t\}\times dx).\]

Also, the random measure $W\cdot \mu$ is defined by
$$
(W\cdot \mu  )(\omega  ,dt\times dx) = W(\omega  ,t,x)\mu  (\omega ,dt\times dx).
$$
\begin{definition}\label{Def15.24}
The random measure $\mu  $ is \emph{optional} (resp.
\emph{predictable}) if for each positive optional (resp. predictable) process $W$ the processes $W*\mu  ^+$ and $W*\mu  ^-$ are
optional (resp. predictable).

The \emph{natural filtration of $\mu$} is the smallest right-continuous complete filtration in which $\mu$ is optional.
\end{definition}

Define the following sets of (equivalence classes of) random measures
\begin{itemize}
 \item  $\tilde{\CA}$ will denote the set of optional and integrable random measures.
\item $\tilde \CA_\sigma$ will denote the set of random measures $\mu$ for which there
exists a $\tilde \Sigma_p$ measurable partition $\{B_n\}_{n\in\bN}$ of $\tilde \Omega$ such that $I_{B_n}\cdot \mu  \in
\tilde\CA$ for each $n.$
\end{itemize}

\begin{theorem}
\index{Notation!$\mu_p$}
Suppose $\mu  \in \tilde\CA_\sigma  \,.$ There is then a unique
predictable random measure $\mu_p\in \tilde\CA_\sigma  $ such that,  for every $\tilde\Sigma_p$-measurable function $W$ such that
$W*\mu  \in \CA_{\loc}$ we have 
\[\uPi^*_p(W*\mu  ) = W*\mu_p,\]
where $\uPi^*_p$ denotes the dual predictable projection in $\CA_\loc$ (the processes of locally integrable variation). The measure $\mu_p$ is called the \emph{dual predictable projection} or \emph{compensator} of
$\mu.$ 
\end{theorem}

\begin{definition}\label{Def15.40}
\index{Terminology!Integer valued random measure}
A random measure $\mu$ is said to be \emph{integer valued} if
\begin{enumerate}[(i)]
\item  $\mu  (\omega  ,\{t\}\times \CE)\le 1$ for all $(\omega  ,t),$
\item  for every $A\in \CB([0,\infty])\otimes \FE,$ the random variable $\mu(\omega, A)$ takes
values in the set $\overline{\mathbb{Z}}^+= \{0,1,2,\dots\}\cup \{\infty  \}.$ 
\end{enumerate}
 Note that $\mu$ is then a nonnegative random measure. We denote by $\tilde \CA^1_\sigma$ the integer valued elements of $\tilde\CA_\sigma$.
\end{definition}

\begin{definition}
 For $\mu\in \tilde\CA^1_\sigma$, let $\tilde \mu = \mu-\mu_p$. Then $\tilde\mu$ is a \emph{martingale random measure}.

We say $W$ is stochastically integrable with respect to $\tilde \mu$ if it is $\tilde\Sigma_p$-measurable and
\[\Big(\sum_{t\leq (\cdot)} \Big(\int_{\CE} W(\omega, t,z) \tilde\mu(\omega, \{t\}\times dz)\Big)^2\Big)^{1/2} \in \CA_{\loc}.\]
 In this case, the stochastic integral of $W$ with respect to $\tilde \mu$ is (uniquely) defined to be the purely discontinuous local martingale $X$ such that 
\[\Delta X_t(\omega) = \Delta (W*\tilde\mu)_t = \int_{\CE} W(\omega, t,z) \tilde\mu(\omega, \{t\}\times dz)\]
up to indistinguishability.
\end{definition}

\begin{theorem}\label{thm:QVrandommeasure}
 For $\mu\in \tilde\CA^1_\sigma$, let $\tilde \mu = \mu-\mu_p$. We define the optional and predictable quadratic variation measures
\[\begin{split}
[ \tilde \mu](\omega, [0,t]\times B) = \sum_{s\in[0,t]}\big(\mu(\omega,  \{s\}\times B)-\mu_p(\omega,  \{s\}\times B)\big)^2,\\
\la \tilde \mu\ra(\omega, [0,t]\times B) = \mu_p(\omega, [0,t]\times B) - \sum_{s\in[0,t]} \big(\mu_p(\omega, \{s\}\times B)\big)^2.\end{split}
\]
Then $\la \tilde\mu\ra$ is a predictable nonnegative random measure, and for any  $W$ stochastically integrable with respect to $\tilde\mu$ we have
\[[ W*\tilde \mu] = W^2* [ \tilde \mu] \q\text{ and }\q\la W*\tilde \mu\ra = W^2* \la \tilde \mu\ra.\]
\end{theorem}

\section{The single jump measure}

A particularly straightforward example of a random measure is given by a process with a single jump. This basic case forms the starting point for our approach.

\begin{definition}
The single-jump measure $\mu$ is defined by a pair $(T,Z)$, where $T\in[0,\infty]$, $Z\in\CE$ and
\[\mu(\omega, [0,t]\times A)=I_{\{T(\omega)\leq t\}}I_{Z(\omega)\in A}.\]
\end{definition}

We consider the filtration $\{\F_t\}_{t\geq 0}$, where $\F_0$ is arbitrary and 
\[\F_t =\sigma(I_{T\leq t}, ZI_{T\leq t}) \vee \F_0.\] We can associate the probability measure $\bP$ with a deterministic measure $\nu$ (or more generally, an $\F_0$-measurable random measure $\nu$) on $[0,\infty]\times A$, defined by 
\[\nu(\omega, [t,\infty  ]\times A) = \bP(T\geq t, Z\in A |\F_0).\]
For $A\in \FE$ write
$$
F^A_t = \nu(\omega, [t,\infty  ]\times A),
$$
so that $F^A_t$ is the probability (given $\F_0$) that $T>t$ and $Z\in A.$ 
Furthermore, write $F_t = F^\CE_t$ and $c=\inf\{t:F_t=0\}.$

\begin{lemma}\label{Lem15.2}
Suppose $\tau  $ is an $(\CF_t)$ stopping time. Then there is a $\F_0$-measurable 
$t_0\in [0,\infty]$ such that $\tau  \wedge T = t_0 \wedge T\;\as.$
\end{lemma}

In this context, we have a simple version of a martingale representation theorem, due to Chou and Meyer \cite{Chou1975} (see also Davis \cite{Davis1976} and Elliott \cite{Elliott1976, Elliott1977}). We present this argument as it will form the basis for our approach.

\begin{theorem}[Chou and Meyer \cite{Chou1975}]\label{Thm15.14}
An adapted process $M$ is a $\{\CF_t\}_{t\geq0}$-local martingale if and
only if there exists some $g$ such that 
\[M_t =M_0+\int_{]0,t]\times\CE}  g(s,x)\tilde\mu(ds,dx).\]
Furthermore, this function $g$ satisfies
\[E[|g(T, z)|I_{\{T<t\}}]<\infty\quad \text{for all $t$ with $F_t>0$},\]
and is unique up to equality $d\la \tilde\mu\ra\times d\bP$-a.e.
\end{theorem}

\begin{proof}
Suppose $E[|g(T, z)|I_{T<t}]<\infty$ for all $t$ with $F_t>0$. Then straightforward calculations show that $M_0+(g*\tilde\mu)$ is a local $\{\CF_t\}$-martingale. 

To show the converse, by localization we can suppose that $M$ is a uniformly integrable martingale. Without loss of generality, assume $M_0=0$. The Doob--Dynkin lemma implies that there exists an $\F_0$-measurable function $h$ such that $M_\infty = h(T,z)$, and hence
\[
M_t = h(T,z)I_{\{t\ge T\}}- I_{\{t<T\}} \frac{1}{F_t} \int_{]0,t]\times \CE}h(s,x)d\nu (s,x),
\]
Now, for any sufficiently integrable $g$,
\[
(g*\tilde\mu)_t=\int_{]0,t]\times\CE}g\,d\tilde\mu= I_{\{t\ge T\}} g(T,z) - \int_{]0,t\wedge T]\times \CE} g(s,x) \frac{1}{F_{s-}}d\nu.
\]
Define,
\[
g(t,x) = h(t,x) + \frac{1}{F_t} \,\int_{]0,t]\times \CE} h(s,x')d\nu(s,x'),
\]
and Fubini's theorem yields
\[\begin{split}
&\int_{]0,t]\times \CE} g(s,x)\frac{1}{F_{s-}}d\nu  \\
&= \int_{]0,t]\times \CE}  h(s,x)\frac{1}{F_{s-}}d\nu  (s,x) -\int_{]0,t]} \Big(\frac{1}{F_sF_{s-}} \int_{]0,s]\times \CE}h(u,x)\nu  (du,dx)\Big)dF_s\\
&= \int_{]0,t]\times \CE} \frac{h}{F_{s-}}d\nu + \int_{]0,t]\times \CE} \Big(-\int_{[u,t]} \frac{1}{F_sF_{s-}}dF_s\Big) h(u,x)d\nu\\
&= \int_{]0,t]\times \CE} \frac{h}{F_{s-}}d\nu  +\int_{]0,t]\times E}\Big(\,\frac{1}{F_t}\,-\,\frac{1}{F_{u-}}\,\Big) h(u,x)d\nu\\
&= \,\frac {1}{F_t}\,\int_{]0,t]\times \CE} h\,d\nu .
\end{split}
\]
Therefore $M=M_0+(g*\tilde\mu)$, as desired. 

It remains to show that $g$ has the desired integrability and is unique. To do this, we note that if $E[|M_t|]<\infty$ then our construction gives
$$
\begin{aligned}
\int_{]0,t]\times\CE}   \vert  g\vert  d\nu
&\le \int_{]0,t]\times\CE} \vert  h\vert  d\nu  - \int_{]0,t]}\,\Big(\frac{1}{F_s}\,\int_{]0,s]\times E} \vert  h\vert  d\nu\Big) dF_s\\
&\le \int_{]0,t]\times\CE} \vert  h\vert  d\mu   - F^{-1}_{t}\,\int_{]0,t]}\,\int_{]0,s]\times E} \vert  h\vert  d\nu dF_s\\
&= \int_{]0,t]\times\CE} \vert  h\vert  d\mu  + F^{-1}_{t}\,\int_{]0,t]\times E} (F_s - F_{t}) \vert  h\vert  d\nu  \\
&\le \Big(1+\frac{1}{F_{t}}\Big)\int_{]0,t]\times\CE}\vert  h\vert  d\nu\\
&=\Big(1+\frac{1}{F_{t}}\Big)E[|M_T| I_{\{T\leq t\}}|\F_0]\\
\end{aligned}
$$
As $E[E[|M_T| I_{\{T\leq t\}}|\F_0]] \leq E[|M_t|] <\infty$, this implies that $E[|g(T,x)|I_{\{T<t\}}]<\infty$ for each $t$ such that $F_t>0$.

To see uniqueness, we see that if 
\[M_t = M_0 + \int_{]0,t]} g \,d\tilde\mu = M_0 + \int_{]0,t]} g' \,d\tilde\mu\]
then 
\[\int_{]0,t]} (g-g') d\tilde\mu = 0\text{ and so }E\Big[\int_{]0,t]} (g-g')^2 d\la\tilde\mu\ra\Big] = 0.\]
\end{proof}

\section{Well-ordered jump case}
Now suppose that $\mu$ is an integer valued random measure. Let $\{\F_t\}_{t\geq 0}$ be the natural filtration of $\mu$. Write $\tilde\mu=\mu-\mu_p$ for the martingale random measure generated by $\mu$.  Suppose that the jumps of $\mu$ are well ordered, that is, for any stopping time $S$  there exists a stopping time $T>S$ such that 
\[\mu(\omega, \rrbracket S, T\llbracket\times\CE)=0,\qquad  \mu(\omega, \llbracket T\rrbracket\times\CE)=1.\]

In order to prove the martingale representation theorem in this setting, we follow Elliott \cite{Elliott1976} in using the following version of the principle of transfinite induction. 

We denote by $\gimel$ the first uncountable ordinal (this is conventionally called $\Omega$ or $\omega_1$, however this leads to confusion in our probabilistic setting). The set $\gimel$ is well ordered and uncountable, and for every $\alpha\in\gimel$ the set $\alpha\leq \beta$ is countable.  The smallest element of $\gimel$ is denoted $0$, the successor to $\alpha\in\gimel$ is denoted $\alpha+1$, and $\alpha$ is called the predecessor of $\alpha+1$. If $\alpha\in\gimel$ does not have a predecessor it is called a limit ordinal, other elements are non-limit ordinals. The following theorem gives the properties of $\gimel$ we will use.

\begin{theorem}[Principle of transfinite induction]
The following can be found in Dellacherie and Meyer \cite[Vol I, Chapter 0]{Dellacherie1975}.
\begin{itemize}
 \item If $f$ is a monotonic increasing function from $\gimel$ into $\extreal$, then there is an ordinal $\alpha\in\gimel$ such that $f(\beta)=f(\gamma)$ for every $\beta\geq \alpha$.
 \item Suppose $P(\alpha)$ is some property of the ordinal $\alpha$, and suppose that
\begin{enumerate}[(i)]
 \item if $P(\alpha)$ is true then $P(\alpha+1)$ is true,
 \item if $\beta$ is a limit ordinal and if $P(\alpha)$ is true for all $\alpha<\beta$ then $P(\beta)$ is true, and
\item $P(0)$ is true.
\end{enumerate}
Then $P(\alpha)$ is true for all $\alpha\in\gimel$.
\end{itemize}
\end{theorem}

We first show the following lemma.
\begin{lemma}\label{lem:transfiniteinductionlemma1}
For a random measure in $\tilde\CA^1_\sigma$ with well ordered jumps, we can write the jumps of $\mu$ as a family of stopping times $\{T_\alpha\}_{\alpha\in\gimel}$ such that $\mu$ is supported on $\cup_{\alpha\in\gimel}\llbracket T_\alpha\rrbracket$. These stopping times satisfy $T_\alpha<T_\beta$ on the set $\{T_{\alpha}<\infty\}$ for any $\alpha<\beta$ in $\gimel$, and for each $\omega$, there exists an $\alpha\in\gimel$ such that $T_\alpha(\omega)=\infty$.
\end{lemma}
\begin{proof}
 Define $T_0=0$. Under the conditions of the theorem, if $T_\alpha$ is defined then we can find a stopping time $T_{\alpha+1}$ satsifying the conditions. If $\alpha$ is a limit ordinal, then we define $T_\alpha = \sup_{\beta<\alpha} T_\beta$, which is a stopping time as the set $\beta<\alpha$ is countable. Note that as $\alpha \mapsto T_\alpha(\omega)$ is a monotonic increasing function from $\beta$ into $\extreal$ for each $\omega$, then there is an ordinal $\alpha\in\gimel$, depending on $\omega$,  such that $T_\beta(\omega)=T_\alpha(\omega)$ for every $\beta\geq \alpha$. As $T_\beta>T_\alpha$ for all $\beta>\alpha$ unless $T_\alpha=\infty$, we see that $T_\alpha=\infty$.
\end{proof}

\begin{theorem}[Elliott \cite{Elliott1976}]\label{thm:MRTfinitejump}
In the setting of a random measure with well-ordered jumps, any $\{\F_t\}_{t\geq0}$-local martingale $M$ has a representation
\[M_t = M_0 + \int_{]0,t]\times \CE} H(\omega, s, x) \tilde\mu(ds, dx)\]
for some predictable, locally $\mu_p$-integrable function $H$.
\end{theorem}

\begin{proof}
By localizing, we can assume that $M$ is a uniformly integrable martingale. For notational simplicity, we write $H_t$ for $H(\omega, t, z)$.  We will show that there is a predictable function $H$ such that for any $\alpha\in\gimel$,
\begin{equation}\label{eq:MRTtransfiniteinduction1}
 M_{T_\alpha} = M_0 + \int_{]0,T_\alpha]\times\CE} H_t d\tilde\mu.
\end{equation}
This is clear for $\alpha=0$.  We now work using transfinite induction. 

Suppose that we can establish (\ref{eq:MRTtransfiniteinduction1}) for a given $\alpha$. Then consider the filtration $\tilde\F_t=\F_{t+T_\alpha}$. By Theorem \ref{Thm15.14} we can find $H$ such that
\[ M_{T_{\alpha+1}} = M_{T_\alpha} +  \int_{]T_\alpha,T_{\alpha+1}]\times\CE} \bar H_t d\tilde\mu.\]
Defining $\tilde H = I_{[0, T_{\alpha}]}H + I_{]T_\alpha, T_{\alpha+1}]} \bar H$,  rearrangement yields
\[ M_{T_{\alpha+1}} = M_0 +  \int_{]0,T_{\alpha+1}]\times\CE} \tilde H_t d\tilde\mu,\]
so (\ref{eq:MRTtransfiniteinduction1}) holds for $\alpha+1$.

Now suppose that we can establish (\ref{eq:MRTtransfiniteinduction1}) for all $\alpha<\beta$, where $\beta$ is a limit ordinal. As $M$ is assumed to be uniformly integrable, by martingale convergence we have (the convergence being almost sure),
\[
   M_{T_\beta-} = \lim_{\alpha\uparrow \beta}M_{T_\alpha} = \lim_{\alpha\uparrow\beta} \Big(M_0+ \int_{]0,T_\alpha]\times\CE} H_t d\tilde\mu\Big) = M_0 + \int_{]0,T_\beta[\times\CE} H_t d\tilde\mu
\]
for some $H:{\rrbracket 0,T_\beta\llbracket}\times\CE \to \bR$. Furthermore, by the Doob--Dynkin lemma we know that $M_{T_\beta} = M_{T_\beta-} + g(\omega, Z_{T_\beta}(\omega))$, where $Z_{T_\beta}(\omega)$ is the unique value in $\CE$ such that  $\mu(\omega, \{T_\beta\}\times\{Z_{T_\beta}(\omega)\})=1$, if such a value exists, and $Z_{T_\beta}(\omega) = \emptyset_\CE$ otherwise (where $\emptyset_\CE$ denotes a value not in $\CE$), and $g$ is some $\CF_{T_{\beta}-}\otimes (\FE\vee \sigma(\{\emptyset_\CE\}))$-measurable function. As $T_\beta$ is predictable for $\beta$ a limit ordinal,  $E[M_{T_\beta}|\F_{T_{\beta-}}] = M_{T_{\beta}-}$, so $\int_{\CE} g(\omega, z) \mu_p(\omega, \{T_\beta\}\times dz)=0$. Therefore
\[ M_{T_\beta} = M_{T_{\beta}-} + \int_{\CE} g(\omega, z) \mu (\omega, \{T_\beta\}\times dz)= M_{T_{\beta}-} + \int_{\CE} g(\omega, z) \tilde \mu (\omega, \{T_\beta\}\times dz)\]
which implies, writing $H_{T_\beta}=g$,
\[M_{T_\beta} =  M_0 + \int_{]0,T_\beta]\times\CE} H_t d\tilde\mu.\]
By the principle of transfinite induction, we therefore know that (\ref{eq:MRTtransfiniteinduction1}) holds for all $\alpha\in\gimel$. As mentioned, for each $\omega$ there exists an $\alpha$ such that $T_\beta =\infty$ for all $\beta>\alpha$, so we have 
\[  M_\infty = M_0 + \int_{]0,\infty]\times\CE} H_t d\tilde\mu.\]
Stopping $M$ at time $t$ and taking a conditional expectation, we obtain the desired representation.
\end{proof}

\section{Parochial random measures}\label{sec:Parochial}

We now focus on a specific subclass of the random measures, the `parochial' measures, for which we can prove a martingale representation theorem holds. 
Rather than attempt an explicit definition of the class, we give a recursive definition, and then provide some special cases of measures in the class. 

\begin{definition}
Consider the  optional random measures  $\mu\in\tilde\CA^1_\sigma$ in a filtration $\{\F_t\}_{t\geq 0}$, and denote by $\mu_p$ the compensator of $\mu$ in this filtration. We say $\mu$ is parochial if it can be obtained through finitely many applications of the following recursive definition.
\begin{enumerate}[(P1)]
 \item Any random measure with well ordered jumps is parochial.
 \item A random measure is parochial whenever there exists an increasing family of sets $\{A_n\}_{n\in\bN}\subset \tilde\Sigma_p$ such that $\cup_n A_n = [0,\infty]\times\CE$ and, for every $n$, 
\begin{itemize}
 \item $I_{A_n} \cdot \mu$  is parochial in its natural filtration, and 
\item  the compensator of $I_{A_n} \cdot \mu$  in its natural filtration is $I_{A_n} \cdot \mu_p$.
\end{itemize}

 \item Any `right-parochial' random measure is parochial. A random measure is `right-parochial' if for any stopping time $S$ there exists a stopping time $T$ with $T>S$ on the set $\{S<\infty\}$ such that
\[\mu^{(S)}(\omega, [0,t]\times dz) := \mu(\omega, [S, (t+S)\wedge T]\times dz)\]
is parochial in the filtration $\{\tilde \F^{(S)}_t = \F_{(t+S)\wedge T}\}_{t\geq 0}$.
\end{enumerate}
\end{definition}

\begin{remark}
 Clearly we could replace (P1) with the statement `a random measure with a single jump is parochial', and then obtain the well-ordered jump case by applying (P3). From the result of Lemma \ref{lem:parochialexamples} below, we can also replace (P1) with the statement `the random measures in $\tilde\CA^1_\sigma$ with \emph{$\F_0$-measurable} compensators are parochial', and will obtain the same definition by applying (P3).
\end{remark}
\begin{remark}
 This definition is, in some sense, `open', as it is motivated by the methods we have at our disposal for extending martingale representation theorems, as we will see in Theorem \ref{thm:MRTparochial}. If further extension results are available, the definition of parochial can be augmented to allow these cases.
\end{remark}

We now give some useful examples of parochial measures, which demonstrates that this is a richer class of random measures than it may first seem. The `locality' implicit the the examples (iii) and (v) is what suggested the term `parochial'. In practice, example (ii), coupled with the recursive definition, allows many natural examples of parochial measures to be obtained.

\begin{lemma}\label{lem:parochialexamples}
The following special cases  are parochial.
\begin{enumerate}[(i)]
\item Any random measure with well ordered jumps in its natural filtration.
\item Any random measure with deterministic compensator (or more generally with $\F_0$-measurable compensator). This includes all jump processes with independent increments.
\item Any random measure where $\mu_p(\omega, [0,t]\times dx)$ is $\F_{t-\epsilon}$-measurable, for some fixed $\epsilon$ and all $t$.
\item More generally any random measure with right-deterministic compensator, that is, where for every $\{\F_t\}_{t\geq0}$-stopping time $S$ there exists an $\{\F_t\}_{t\geq0}$-stopping time $T>S$ and a random measure $\nu^S$ which is $\F_S\otimes \CB([0,\infty])\otimes\FE$-measurable such that $I_{\llbracket S, T\rrbracket}\cdot \mu_p=I_{\llbracket S, T\rrbracket}\cdot \nu^S$
\item Any random measure where there exists a partition $\{B_n\}_{n\in\bN}$ of $\CE$ such that  $\mu(\omega, [0,t]\times B_n)$ is parochial, and $I_{B_n}\cdot\mu(\omega, [0,t]\times dx)$ is independent of $I_{B_m}\cdot \mu(\omega, [0,t]\times dx)$ for all $m\neq n$. 
\end{enumerate}
\end{lemma}
\begin{proof}
 Case (i) is trivial. Case (iii) follows from (iv). We prove the remaining cases.

\textbf{Case (ii).}
As we have assumed $\mu\in\CA^1_\sigma$, there exists an $\F_0$-measurable increasing sequence of  sets $\{A_n\}_{n\in\bN}\subset \CB([0,\infty])\otimes \FE$ which almost surely partition $[0,\infty]\times \CE$ and $\mu_p(A_n)<\infty$. This implies that $\mu$ almost surely has finitely many jumps in $A_n$, and so these jumps are well ordered. Hence by (P2) of the definition of parochial measures, as measures with well ordered jumps are parochial (P1),  we see $\mu$ is parochial.

\textbf{Case (iv).}
This is simply an application of (P3) of the definition of parochial measures applied to the $\F_0$-measurable measures considered in Case (ii). 

\textbf{Case (v).} The independence of the jumps implies the compensator of $I_{B_n}\cdot\mu(\omega, [0,t]\times dx)$ is predictable in the natural filtration of $I_{B_n}\cdot\mu(\omega, [0,t]\times dx)$, so the result follows by applying (P2) with $A_n=\cup_{m\leq n}B_m$.
\end{proof}

We can now present a martingale representation theorem for parochial random measures. We write $\CH^p$ for the family of martingales with $\|M\|_{\CH^p} := E[\sup_t |M_t|^p]^{1/p}<\infty$, for $p\geq 1$. We begin with a definition and preliminary lemma.

\begin{definition}
 We say that a local martingale or martingale random measure has the predictable representation property if every $\{\F_t\}_{t\geq 0}$-local martingale  $M$ can be written in the form 
\[M_t = M_0 + \int_{]0,t]\times\CE} H_t d\tilde\mu\]
for some $\tilde\mu$-stochastically integrable function $H$.
\end{definition}

\begin{lemma}\label{lem:approxterminal}
 Suppose that for every $B\in\F_\infty$ we can write
\[E[I_B|\F_t] = E[I_B|\F_0] + \int_{]0,t]\times\CE} H_t d\tilde\mu\]
for some $\tilde\mu$-stochastically integrable function $H$. Then $\tilde\mu$ has the predictable representation property.
\end{lemma}

\begin{proof}
First suppose $M$ is a uniformly integrable martingale and $M_\infty\in L^{1+\epsilon}$ for some $\epsilon>0$. Then there exists a sequence of simple random variables $M_\infty^{(n)}$ (ie of the form $M_\infty^{(n)}=\sum_{k=1}^n c_k I_{B_k}$) which converge to $M_\infty$ in $L^{1+\epsilon}$. By linearity of the integral, every simple function $M_\infty^{(n)}$ has the desired representation in terms of an integrand $H^n$, and by the BDG inequality and Doob's $L^p$ inequality,
\[E\Big[\Big(\int_{]0,\infty]\times\CE} (H_t^n- H_t^{m})^2 d[\tilde\mu]\Big)^\frac{1+\epsilon}{2}\Big]^{1/(1+\epsilon)}\leq K\|M_\infty^{(n)}-M_\infty^{(m)}\|_{L^{1+\epsilon}} \to 0,\]
for some constant $K$ depending on $\epsilon$. Therefore there exists a limiting process $H$, which is stochastically integrable with respect to $\tilde\mu$, such that 
\[M_\infty =  M_0 + \int_{]0,\infty]} H_t d\tilde\mu.\]

Now suppose $\sup_t|M_t| \in L^1$, so $M\in\CH^1$. For any $\epsilon>0$, as $\CH^{1+\epsilon}$ is dense in $\CH^1$ (indeed from Dellacherie and Meyer \cite[VII.3, Theorem 71]{Dellacherie1975} the bounded martingales are dense in $\CH^1$), we can find a sequence $M^{(n)}$ such that $\|M^{(n)} - M\|_{\CH^1} \to 0$ and $M^{(n)}_\infty\in L^{1+\epsilon}$. Therefore, from the BDG inequality, as $M^{(n)}$ has the desired representation for each $n$, there exists a constant $K$ such that 
\[E\Big[\Big(\int_{]0,\infty]\times\CE} (H_t^n- H_t^{m})^2 d[\tilde\mu]\Big)^\frac{1}{2}\Big]\leq K\|M^{(n)} - M^{(m)}\|_{\CH^1} \to 0,\]
and so an appropriate limit $H$ exists. That $H$ is stochastically integrable with respect to $\tilde\mu$ follows from the fact $M$ is a semimartingale, so $[M]^{1/2}\in\CA_\loc$.

Finally, if $M$ is a local martingale, then the stopping times $T_n= \inf\{t: |M_{t}|\geq n\}$ are a localizing sequence with $\sup_{t} |M_{t\wedge T_n}|\leq n+|\Delta M_{T_n}| \in L^1$. We can therefore write
\[M_{t\wedge T_n} = M_0 + \int_{]0,t\wedge T_n]} H_s d\tilde\mu\]
for each $n$, and pasting yields the desired representation.
\end{proof}

\begin{remark}
 From the BDG inequality, we also see that this representation is unique, in the sense that if $M=a+ (g*\tilde \mu) = b+(h*\tilde\mu)$, then $a=b=M_0$ and for $T_n$ a localizing sequence such that the stopped processes satisfy $M^{T_n}\in\CH^1$, we have
\[E\Big[\Big(\int_{]0,T_n]\times\CE} (g_t- h_t)^2 d[\tilde\mu]\Big)^\frac{1}{2}\Big] =0\]
for all $n$, and hence $E\big[\big((g_t- h_t)^2 * \la\tilde\mu\ra\big)_\infty\big] =0$.
\end{remark}

\begin{theorem}\label{thm:MRTparochial}
 Let $\mu\in \tilde\CA^1_\sigma$ be a parochial random measure in its natural filtration $\{\F_t\}_{t\geq 0}$. Write $\tilde\mu=\mu-\mu_p$. Then $\tilde\mu$ has the predictable representation property.
\end{theorem}
\begin{proof}
 We proceed inductively, using the recursive nature of the definition of parochial measures as a guide.

\textbf{(1)} Theorem \ref{thm:MRTfinitejump} states that if the random measure has well ordered jumps, then it has the predictable representation property.

\textbf{(2)} By Lemma \ref{lem:approxterminal}, it is enough to prove the theorem under the assumption that $M_\infty = I_B$ for some $B\in\F_\infty$. Suppose that there exists an increasing family of sets $\{A_n\}_{n\in\bN}\subset \tilde\Sigma_p$ such that $\cup_n A_n = [0,\infty]\times\CE$ and, for every $n$, 
\begin{itemize}
 \item $I_{A_n} \cdot \tilde\mu$ has the predictable representation property in its natural filtration and
 \item the compensator of $I_{A_n} \cdot \mu$  in its natural filtration is $I_{A_n} \cdot \mu_p$.
\end{itemize}
Write $\{\F^n_t\}_{t\geq0}$ for the natural filtration of $I_{A_n}\cdot \mu$. 

If $B\in \F^n_\infty$ for some $n$, we have a representation for the martingale $M=E[I_B|\F^n_{(\cdot)}]$. As $I_{A_n}\cdot \mu_p$ is the compensator of $I_{A_n}\cdot \mu$ in $\{\F^n_t\}_{t\geq0}$, this representation is in terms of the $\{\F_t\}_{t\geq0}$-martingale random measure $\mu-\mu_p$, so we know that $M=E[I_B|\F_{(\cdot)}]$.

Let $\CC$ be the family of all sets $B\in\CF_\infty$ such that $M=E[I_B|\F_{(\cdot)}]$ has a representation of the desired form. From our earlier comments, $\CC$ contains the algebra of sets $\cup_{n\in\bN} \F^n_\infty$ (this is an algebra as $\F^n_\infty$ is increasing in $n$). If $\{B_m\}_{m\in\bN}$ is an increasing sequence in $\CC$ with $B=\cup_m B_m$, then $I_{B_m} \to I_B$ in $L^2$, by monotone convergence. Considering the corresponding representations, 
\[I_{B_m} = M_0 + \int_{]0,\infty]\times\CE} H_t^m d\tilde\mu\]
from It\=o's isometry we have
\[\begin{split}E[(I_{B_m} - I_{B_{m'}})^2] &= E\Big[\Big(\int_{]0,\infty]\times\CE} H_t^m- H_t^{m'} d\tilde\mu\Big)^2\Big]\\& = E\Big[\int_{]0,\infty]\times\CE} (H_t^m- H_t^{m'})^2 d\la\tilde\mu\ra\Big]\end{split}\]
so $H_t^m$ converges in the Hilbert space $L^2(\la\tilde\mu\ra)$ to a process $H$. Hence there exists the limit
\[I_B = M_0 + \int_{]0,\infty]} H_t d\tilde\mu,\]
and we see that $B\in \CC$. Similarly if $\{B_m\}_{m\in\bN}$ is a decreasing sequence we have $\cap_m B_m \in \CC$. Therefore $\CC$ is a monotone class.

The monotone class theorem then states that $\CC$ contains the $\sigma$-algebra generated by $\cup_n \F^n_\infty$, that is, $\F_\infty$. Therefore, for any $B\in \F_\infty$ we have the desired representation of $M=E[I_B|\F_{(\cdot)}]$. By Lemma \ref{lem:approxterminal}, this implies that $\tilde\mu$ has the predictable representation property.

\textbf{(3)} Suppose that for any stopping time $S$, there exists a stopping time $T$ with $T>S$ on the set $\{S<\infty\}$ such that
\[\mu^{(S)}(\omega, [0,t]\times dz) := \mu(\omega, [S, (t+S)\wedge T]\times dz)\]
has the predictable representation property in the filtration $\{\tilde\F_t^{(S)}=\F_{(t+S)\wedge T}\}_{t\geq0}$.

Then just as in the proof of Theorem \ref{thm:MRTfinitejump}, we can use transfinite induction to establish the result. As in Lemma \ref{lem:transfiniteinductionlemma1}, there exists a family of stopping times $\{T_\alpha\}_{\alpha\in\gimel}$ such that
\begin{itemize}
 \item $\mu^{(T_\alpha)}(\omega, [0,t]\times dz) := \mu(\omega, [T_\alpha, (t+T_\alpha)\wedge T_{\alpha+1}]\times dz)$ has the predictable representation property in the filtration $\{\tilde\F_t^{(T_\alpha)}=\F_{(t+T_\alpha)\wedge T_{\alpha+1}}\}_{t\geq0}$ 
\item $T_\alpha<T_\beta$ for all $\alpha<\beta$ on the set $\{T_\alpha<\infty\}$, and
\item there exists $\alpha\in\gimel$ such that  $T_{\beta}=\infty$ for all $\beta\geq \alpha$.
\end{itemize}
 Therefore, as the representation holds in $\{\tilde\F_t^{(T_\alpha)}\}_{t\geq0}$, we know that we can write
\[M_{T_{\alpha+1}} = M_{T_\alpha} + \int_{]T_\alpha, T_{\alpha+1}]} H_s d\mu\]
for some predictable function $H$. As in the proof of Theorem \ref{thm:MRTfinitejump}, given the representation up to $T_\alpha$ for all $\alpha<\beta$, where $\beta$ is a limit ordinal, we can use martingale convergence to construct the representation up to $T_\beta$. Transfinite induction establishes existence of a representation for all times. (As these arguments are only trivial modifications of the proof of Theorem \ref{thm:MRTfinitejump}, the full details are left to the reader.)

As this establishes that the recursions in the definition of parochial measures preserve the predictable representation property, the theorem is proven.
\end{proof}

\section{A counterexample}
One might wonder whether the assumption that the random measure is parochial is truly needed, that is, whether \emph{all} random measures in $\tilde\CA^1_\sigma$ yield martingale random measures with the predictable representation property. This is an intuitively reasonable proposition, as all random measures with independent increments have the predictable representation property (Lemma \ref{lem:parochialexamples} and Theorem \ref{thm:MRTparochial}), and all finite-activity random measures have the preditctable representation property (Theorem \ref{thm:MRTfinitejump}). It is, however, false, as the following example will show.

Let $\CE=\{1/n\}_{n\in\bN}$. Let $W$ be a Brownian motion in its natural filtration and $\mu$ be a random measure in $\tilde\CA^1_\sigma$ such that, in the (right-continuous, complete) filtration generated by $W$ and $\mu$, 
\[\mu_p(dt\times\{1/n\}) = ne^{W_t}dt.\]
Now consider the natural filtration of $\mu$, denoted  $\{\CF_t\}_{t\geq0}$,  and let $\nu$ be the compensator of $\mu$ in this filtration. If $\mu-\nu$ has the predictable representation property, then all $\{\CF_t\}_{t\geq0}$-local martingales must be purely discontinuous. Therefore, $W$ must not be $\{\CF_t\}_{t\geq0}$-adapted. We will show that this leads to a contradiction.

By construction and Theorem \ref{thm:QVrandommeasure}, we know
\[\begin{split} E[\mu(\omega, &]t,t+h]\times\{1/n\})|W_t]\\
&= E\Big[\int_{]t,t+h]\times\{1/n\}}d\mu_p\Big|W_t\Big] = E\Big[\int_{]t,t+h]}ne^{W_s} ds\Big|W_t\Big]\\& = ne^{W_t}\int_{]t,t+h]}E[e^{W_s-W_t}|W_t] ds = ne^{W_t}\int_{]0,h]}e^{s^2/2} ds \\
&= nh(1+o(h^2))e^{W_t}\\
\text{Var}[\mu(\omega, &]t,t+h]\times\{1/n\})|W_t]\\& = E\Big[\int_{]t,t+h]\times \{1/n\}}d\la \mu-\mu_p\ra\Big|W_t\Big] = nh(1+o(h^2))e^{W_t}
\end{split}\]
Chebyshev's inequality then implies that,
\[\bP\bigg(\Big|\frac{\mu(\omega, ]t,t+h]\times\{1/n\})}{nh} - (1+o(h^2))e^{W_t}\Big| \geq \epsilon\bigg) \leq \frac{(1+o(h^2))E[e^{W_t}]}{nh\epsilon^2}.\]
As $e^{W_t}$ is integrable, taking a limit in probability with respect to $n$ we have
\[\lim_{h\to 0} \lim_{n\to \infty} \frac{\mu(\omega, ]t,t+h]\times\{1/n\})}{nh} = \lim_{h\to 0}(1+o(h^2))e^{W_t} = e^{W_t}.\]
Therefore, as our filtration is assumed to be right-continuous, we know $e^{W}$, and hence $W$, is adapted to $\{\CF_t\}_{t\geq0}$, giving a contradiction.

\section{Including diffusion terms}
Using similar techniques, we can obtain the following representation in spaces which also include continuous martingales.

\begin{definition}
 We say that a continuous martingale $Y$ has right-deterministic volatility if for every stopping time $S$ there exists a stopping time $T$ with $T>S$ on the set $\{S<\infty\}$  and an increasing $\F_S$-measurable map $\Gamma^{(S)}(\omega):\bR\to\bR$ such that $I_{\llbracket S, T\rrbracket}\la Y \ra = I_{\llbracket S, T\rrbracket} \Gamma^{(S)}$. 
\end{definition}
\begin{theorem}\label{thm:MRTrightdeterministicvol}
 If $Y$ is a continuous martingale with right-deterministic volatility, in its natural filtration, then $Y$ has the predictable representation property.
\end{theorem}
\begin{proof}
 As before, we can use transfinite induction. Between the stopping times at which the measurability of the volatility changes, the result follows from the classical representation theorem for Brownian motion, coupled with the Dambis--Dubins--Schwarz theorem. Continuity guarantees that the result holds at every limit ordinal, so the result holds everywhere.
\end{proof}

\begin{remark}
 If we think of a Brownian motion as the limit of a series of jump processes with infinitely many infinitesimally small jumps, the statement that $Y$ has right deterministic volatility can be seen as a  natural limit of the statement that the compensator of the jumps is parochial. By Lemma \ref{lem:parochialexamples}, if the compensator if right-deterministic then it is parochial, which is what we are assuming here.
\end{remark}

\begin{theorem}\label{thm:MRTparochialcharacteristicsofX}
 Let $X$ be a real valued local martingale with  continuous martingale part $X^c$ and associated jump measure $\mu^X$. The semimartingale characteristics of $X$ are given by $\la X^c\ra$ and $\mu^X_p$.

Let $\{\CF_t\}_{t\geq0}$ be the filtration generated by $X$. If $\la X^c\ra$ is right-deterministic and $\mu^X$ is parochial (both in their natural filtrations), then $X$ satisfies a `weak' predictable representation property in its natural filtration,  that is any local martingale can be written
\[M_t = M_0 + \int_{]0,t]} H_s dX^c + \int_{]0,t]\times \bR} G_s d(\mu^X-\mu^X_p),\]
for some appropriate functions $H$ and $G$.
\end{theorem}
We say $X$ has a `weak' predictable representation property as the integral is not with respect to $X$ directly, but with respect to $X^c$ and $\mu^X-\mu^X_p$.
\begin{proof}
By Theorem \ref{thm:MRTrightdeterministicvol} we know $X^c$ has the predictable representation property in its natural filtration, which we denote $\{\F^c_t\}_{t\geq0}$, and by Theorem \ref{thm:MRTparochial}, $(\mu^X-\mu^X_p)$ has the predictable representation property in the natural filtration of $\mu^X$, which we denote $\{\F^d_t\}_{t\geq0}$. By Lemma \ref{lem:approxterminal}, it is enough for us to show the theorem holds when $M_\infty = I_{B}$ for some $B\in \F_\infty$, as all other cases can be approximated by this. We will use a monotone class argument.

Let $\CC$ be the family of sets $B\in \F_\infty$ for which we have the desired representation. Suppose $B=B^c\cap B^d$, for some $B^c\in\F^c_\infty$ and $B^d\in\F^d_\infty$. Then, as we can represent $E[I_{B^c}|\F^c_{(\cdot)}]$ and $E[I_{B^d}|\F^d_{(\cdot)}]$ we know there exists $\alpha^c, \beta^c, \alpha^d, \beta^d$ of appropriate integrability and dimensions that
\[I_{B^c\cap B^d} = I_{B^c}I_{B^d} = \Big(\alpha^c + \int_{]0,t]} \beta^c_s dX^c\Big)\Big(\alpha^d + \int_{]0,t]\times \bR} \beta^d_s d(\mu^X-\mu^X_p)\Big).\]
By It\=o's product rule, as $\int_{]0,t]\times \bR} \beta^d d(\mu^X-\mu^X_p)$ is purely discontinuous and $\int_{]0,t]} \beta^c_s dX^c$ is continuous, 
\[\begin{split}I_{B^c\cap B^d} &= \alpha^c\alpha^d + \int_{]0,t]} \Big(\alpha^d + \int_{]0,t[\times \bR} \beta^d d(\mu^X-\mu^X_p)\Big)\beta^c_s dX^c \\
   &\qquad + \int_{]0,t]\times \bR}\Big(\alpha^c + \int_{]0,t]} \beta^c_s dX^c\Big)\beta^d_s d(\mu^X-\mu^X_p)
  \end{split}
\]
which gives our desired representation. By linearity of the integral we also have a representation for $I_{B^c\cup B^d} = I_{B^c}+ I_{B^d}- I_{B^c\cap B^d}$, and so $\CC$ contains the algebra of sets given by finite intersections and unions of sets in $\F^c_\infty$ and $\F^d_\infty$. As in part (2) of the proof of Theorem \ref{thm:MRTparochial}, the It\=o isometry shows $\CC$ forms a monotone class, so it follows that $\CC$ contains all of $\F^c_\infty\otimes\F^d_\infty=\F_\infty$, as desired.
\end{proof}

\begin{corollary}
Under the assumptions of Theorem \ref{thm:MRTparochialcharacteristicsofX}, any continuous martingale can be written as an integral with respect to $X^c$ and any purely discontinuous martingale can be written as an integral with respect to $\mu^X-\mu^X_p$.
\end{corollary}

\begin{remark}
Further generalizations are possible, where, for example, $\mu^X$ is parochial and $\la X^c\ra$ is `right-adapted to the natural filtration of $\mu^X$', or where $\la X^c\ra$ is right-deterministic and $\mu^X$ satisfies the definition of parochial with `right-deterministic' replaced by `right-adapted to the natural filtration of $X^c$'. A sequential argument, using the representation theorem in the first process to derive the general result, is then possible.  This will allow, for example, a representation for our counterexample to be obtained, in terms of both the random measure $\mu$ and the Brownian motion $W$. 
\end{remark}

\bibliographystyle{plain}
\bibliography{../RiskPapers/General}
\end{document}